\documentclass[11pt, reqno]{amsart}%
\usepackage{amsmath, amstext, amsbsy, amssymb, amscd}
\usepackage{amsmath}
\usepackage{amsxtra}
\usepackage{amscd}
\usepackage{amsthm}
\usepackage{amsfonts}
\usepackage{amssymb}
\usepackage{eucal}
\usepackage{color}
\usepackage{graphicx}%
\newtheorem{theorem}{Theorem}[section]
\theoremstyle{plain}

\newtheorem{lemma}[theorem]{Lemma}
\newtheorem{proposition}[theorem]{Proposition}

\theoremstyle{definition}
\newtheorem{definition}[theorem]{Definition}
\newtheorem{example}[theorem]{Example}

\theoremstyle{remark}
\newtheorem{remark}[theorem]{Remark}

\definecolor{A}{rgb}{.75,1,.75}

\numberwithin{equation}{section}

\newcommand{\al}{\alpha}
\newcommand{\be}{\beta}
\newcommand{\C}{\mathbb C}
\newcommand{\Z}{\mathbb Z}
\newcommand{\Cl}{\mathcal C}
\newcommand{\h}{\mathfrak{h}}
\newcommand{\frakH}{\mathfrak{H}}

\newcommand{\I}{\mathfrak{I}}

\newcommand{\wtd}{\widetilde}
\newcommand{\td}{\tilde}

\newcommand{\aH}{\mathfrak{H}}
\newcommand{\aHC}{\mathfrak{H}^{\mathfrak c}}           
\newcommand{\caH}{\mathfrak{H}^{\sim}}             
\newcommand{\saH}{\mathfrak{H}^-}          
{\vskip-\lastskip\medskip
  \noindent
  {\em #1.}\enspace
  }%
{\qed\par\medskip
  }

\begin{document}

\title[The classical spin affine Hecke algebras]
{Hecke-Clifford algebras and spin Hecke algebras I: the classical
affine type}

\author[Ta Khongsap]{Ta Khongsap}

\author[Weiqiang Wang]{Weiqiang Wang}
\address{Department of Math., University of Virginia,
Charlottesville, VA 22904} \email{tk7p@virginia.edu (Khongsap);
ww9c@virginia.edu (Wang)}

\begin{abstract}
Associated to the classical Weyl groups, we introduce the notion
of degenerate spin affine Hecke algebras and affine Hecke-Clifford
algebras. For these algebras, we establish the PBW properties,
formulate the intertwiners, and describe the centers. We further
develop connections of these algebras with the usual degenerate
(i.e. graded) affine Hecke algebras of Lusztig by introducing a
notion of degenerate covering affine Hecke algebras.
\end{abstract}

\maketitle

\section{Introduction}

\subsection{}
The Hecke algebras associated to finite and affine Weyl groups are
ubiquitous in diverse areas, including representation theories
over finite fields, infinite fields of prime characteristic,
$p$-adic fields, and Kazhdan-Lusztig theory for category $\mathcal
O$. Lusztig \cite{Lu1, Lu2} introduced the graded Hecke algebras,
also known as the degenerate affine Hecke algebras, associated to
a finite Weyl group $W$, and provided a geometric realization in
terms of equivariant homology. The degenerate affine Hecke algebra
of type $A$ has also been defined earlier by Drinfeld \cite{Dr} in
connections with Yangians, and it has recently played an important
role in modular representations of the symmetric group (cf.
Kleshchev \cite{Kle}).

In \cite{W1}, the second author introduced the degenerate spin
affine Hecke algebra of type $A$, and related it to the degenerate
affine Hecke-Clifford algebra introduced by Nazarov in his study of
the representations of the spin symmetric group \cite{Naz}. A
quantum version of the spin affine Hecke algebra of type $A$ has
been subsequently constructed in \cite{W2}, and was shown to be
related to the $q$-analogue of the affine Hecke-Clifford algebra (of
type $A$) defined by Jones and Nazarov \cite{JN}.

\subsection{}
The goal of this paper is to provide canonical constructions of the
degenerate affine Hecke-Clifford algebras and degenerate spin affine
Hecke algebras for all {\em classical} finite Weyl groups, which
goes beyond the type $A$ case, and then establish some basic
properties of these algebras. The notion of spin Hecke algebras is
arguably more fundamental while the notion of the Hecke-Clifford
algebras is crucial for finding the right formulation of the spin
Hecke algebras. We also construct the degenerate covering affine
Hecke algebras which connect to both the degenerate spin affine
Hecke algebras and the degenerate affine Hecke algebras of Lusztig.

\subsection{}

Let us describe our constructions in some detail. The Schur
multiplier for each finite Weyl group $W$ has been computed by Ihara
and Yokonuma \cite{IY} (see \cite{Kar}). We start with a
distinguished double cover $\wtd{W}$ for any finite Weyl group $W$:
\begin{eqnarray}  \label{ext}
1 \longrightarrow \Z_2 \longrightarrow \wtd{W} \longrightarrow W
\longrightarrow 1.
\end{eqnarray}
Denote $\Z_2 =\{1, z\}$. Assume that $W$ is generated by
$s_1,\ldots, s_n$ subject to the relations
$(s_{i}s_{j})^{m_{ij}}=1$. The quotient $\C W^- := \C \wtd{W}/
\langle z+1 \rangle$ is then generated by $t_1,\ldots, t_n$
subject to the relations
$(t_{i}t_{j})^{m_{ij}}=1$ for $m_{ij}$ odd, and
$(t_{i}t_{j})^{m_{ij}}= -1$ for $m_{ij}$ even. In the symmetric
group case, this double cover goes back to I.~Schur \cite{Sch}. Note
that $W$ acts as automorphisms on the Clifford algebra $\Cl_W$
associated to the reflection representation $\h$ of $W$. We
establish a (super)algebra isomorphism
$$
\Phi^{fin}: \Cl_W \rtimes \C W  \stackrel{\simeq}{\longrightarrow}
\Cl_W \otimes \C W^-,
$$
extending an isomorphism in the symmetric group case (due to
Sergeev \cite{Ser} and Yamaguchi \cite{Yam} independently) to all
Weyl groups. That is, the superalgebras $\Cl_W \rtimes \C W$ and
$\C W^-$ are Morita super-equivalent in the terminology of
\cite{W2}. The double cover $\wtd{W}$ also appeared in Morris
\cite{Mo}.

We formulate the notion of degenerate affine Hecke-Clifford
algebras $\aHC_W$ and spin affine Hecke algebras $\saH_W$, with
unequal parameters in type $B$ case, associated to Weyl groups $W$
of type $D$ and $B$. The algebra $\aHC_W$ (and respectively
$\saH_W$) contain $\Cl_W \rtimes \C W$ (and respectively $\C W^-$)
as subalgebras. We establish the PBW basis properties for these
algebras:
$$
\aHC_W \cong \C [\h^*] \otimes \Cl_W \otimes \C W,
 \qquad \saH_W
\cong \Cl [\h^*] \otimes \C W^-
$$
where $\C [\h^*]$ denotes the polynomial algebra and  $\Cl [\h^*]$
denotes a noncommutative skew-polynomial algebra. We describe
explicitly the centers for both $\aHC_W$ and $\saH_W$. The two
Hecke algebras $\aHC_W$ and $\saH_W$ are related by a Morita
super-equivalence, i,e. a (super)algebra isomorphism
$$
\Phi: \aHC_W \stackrel{\simeq}{\longrightarrow} \Cl_W \otimes
\saH_W
$$ which extends the
isomorphism $\Phi^{fin}$. Such an isomorphism holds also for $W$
of type $A$ \cite{W1}.

We generalize the construction in \cite{Naz} of the intertwiners
in the affine Hecke-Clifford algebras $\aHC_W$ of type $A$ to all
classical Weyl groups $W$. We also generalize the construction of
the intertwiners in \cite{W1} for $\saH_W$ of type $A$ to all
classical Weyl groups $W$. We further establish the basic
properties of these intertwiners in both $\aHC_W$ and $\saH_W$.
These intertwiners are expected to play a fundamental role in the
future development of the representation theory of these algebras,
as it is indicated by the work of Lusztig, Cherednik and others in
the setup of the usual affine Hecke algebras.

We further introduce a notion of degenerate covering affine Hecke
algebras $\caH_W$ associated to the double cover $\wtd{W}$ of the
Weyl group $W$ of classical type. The algebra $\caH_W$ contains a
central element $z$ of order $2$ such that the quotient of
$\caH_W$ by the ideal $\langle z+1\rangle$ is identified with
$\saH_W$ and its quotient by the ideal $\langle z-1\rangle$ is
identified with Lusztig's degenerate affine Hecke algebras
associated to $W$. In this sense, our covering affine Hecke
algebra is a natural affine generalization of the central
extension (\ref{ext}). A quantum version of the covering affine
Hecke algebra of type $A$ was constructed in \cite{W2}.

The results in this paper remain valid over any algebraically
closed field of characteristic $p \neq 2$ (and in addition $p \neq
3$ for type $G_2$). In fact, most of the constructions can be made
valid over the ring $\Z[\frac12]$ (occasionally we need to adjoint
$\sqrt{2}$).

\subsection{}
This paper and \cite{W1} raise many questions, including a
geometric realization of the algebras $\aHC_W$ or $\saH_W$ in the
sense of Lusztig \cite{Lu1, Lu2}, the classification of the simple
modules (cf. \cite{Lu3}), the development of the representation
theory, an extension to the exceptional Weyl groups, and so on. We
remark that the modular representations of $\aHC_W$ in the type
$A$ case including the modular representations of the spin
symmetric group have been developed by Brundan and Kleshchev
\cite{BK} (also cf. \cite{Kle}).

In a sequel \cite{KW} to this paper, we will extend the
constructions in this paper to the setup of rational double affine
Hecke algebras (see Etingof-Ginzburg \cite{EG}), generalizing and
improving a main construction initiated in \cite{W1} for the spin
symmetric group. We also hope to quantize these degenerate spin
Hecke algebras, reversing the history of developments from quantum
to degeneration for the usual Hecke algebras.

\subsection{}
The paper is organized as follows. In Section~\ref{sec:finite}, we
describe the distinguished covering groups of the Weyl groups, and
establish the isomorphism theorem in the finite-dimensional case.
We introduce in Section~\ref{sec:clifford} the degenerate affine
Hecke-Clifford algebras of type $D$ and $B$, and in
Section~\ref{sec:spin} the corresponding degenerate spin affine
Hecke algebras. We then extend the isomorphism $\Phi^{fin}$ to an
isomorphism relating these affine Hecke algebras, establish the
PBW properties, and describe the centers of $\aHC_W$ and $\saH_W$.
In Section~\ref{sec:cover}, we formulate the notion of degenerate
covering affine Hecke algebras, and establish the connections to
the degenerate spin affine Hecke algebras and usual affine Hecke
algebras.

{\bf Acknowledgements.} W.W. is partially supported by an NSF
grant.

\section{Spin Weyl groups and Clifford algebras}
\label{sec:finite}

\subsection{The Weyl groups} Let $W$ be an (irreducible)
finite Weyl group with the following presentation:
\begin{eqnarray} \label{eq:weyl}
\langle s_1,\ldots,s_n | (s_is_j)^{m_{ij}} = 1,\ m_{i i} = 1,
 \ m_{i j} = m_{j i} \in \Z_{\geq 2}, \text{for }  i
 \neq j \rangle
\end{eqnarray}
For a Weyl group $W$, the integers $m_{i j}$ take values in
$\{1,2,3,4,6\}$, and they are specified by the following
Coxeter-Dynkin diagrams whose vertices correspond to the generators
of $W$. By convention, we only mark the edge connecting $i,j$ with
$m_{ij} \ge 4$. We have $m_{ij}=3$ for $i \neq j$ connected by an
unmarked edge, and $m_{ij}=2$ if $i,j$ are not connected by an edge.

 \begin{equation*}
 \begin{picture}(150,45) 
 \put(-99,18){$A_{n}$}
 \put(-30,20){$\circ$}
 \put(-23,23){\line(1,0){32}}
 \put(10,20){$\circ$}
 \put(17,23){\line(1,0){23}}
 \put(41,22){ \dots }
 \put(64,23){\line(1,0){18}}
 \put(82,20){$\circ$}
 \put(89,23){\line(1,0){32}}
 \put(122,20){$\circ$}
 \put(-30,9){$1$}
 \put(10,9){$2$}
 \put(74,9){${n-1}$}
 \put(122,9){${n}$}
 \end{picture}
 \end{equation*}
 %
 \begin{equation*}
 \begin{picture}(150,55) 
 \put(-99,18){$B_{n}(n\ge 2)$}
 \put(-30,20){$\circ$}
 \put(-23,23){\line(1,0){32}}
 \put(10,20){$\circ$}
 \put(17,23){\line(1,0){23}}
 \put(41,22){ \dots }
 \put(64,23){\line(1,0){18}}
 \put(82,20){$\circ$}
 \put(89,23){\line(1,0){32}}
 \put(122,20){$\circ$}
 \put(-30,10){$1$}
 \put(10,10){$2$}
 \put(74,10){${n-1}$}
 \put(122,10){${n}$}
 %
 \put(102,24){$4$}
 \end{picture}
 \end{equation*}
%
%
 \begin{equation*}
 \begin{picture}(150,75) 
 \put(-99,28){$D_{n} (n \ge 4)$}
 \put(-30,30){$\circ$}
 \put(-23,33){\line(1,0){32}}
 \put(10,30){$\circ$}
 \put(17,33){\line(1,0){15}}
 \put(35,30){$\cdots$}
 \put(52,33){\line(1,0){15}}
 \put(68,30){$\circ$ }
 \put(75,33){\line(1,0){32}}
 \put(108,30){$\circ$}
 \put(113,36){\line(1,1){25}}
 \put(138,61){$\circ$}
 \put(113,29){\line(1,-1){25}}
 \put(138,-1){$\circ$}
 \put(-29,20){$1$}
 \put(10,20){$2$}
 \put(60,20){$n-3$}
 \put(117,30){$n-2$}
 \put(145,0){$n-1$}
 \put(145,60){$n$}
 %
 \end{picture}
 \end{equation*}
%
%
 \begin{equation*}
 \begin{picture}(150,75) 
 \put(-99,28){$E_{n=6,7,8}$}
 \put(-30,30){$\circ$}
 \put(-23,33){\line(1,0){32}}
 \put(10,30){$\circ$}
 \put(17,33){\line(1,0){32}}
 \put(50,30){$\circ$}
 \put(57,33){\line(1,0){18}}

 \put(81,32){ \dots }
 \put(104,33){\line(1,0){18}}
 \put(122,30){$\circ$}
 \put(129,33){\line(1,0){32}}
 \put(162,30){$\circ$}
 \put(50,-8){$\circ$}
 \put(53,-1){\line(0,1){32}}
 \put(-30,39){$1$}
 \put(10,39){$3$}
 \put(50,39){$4$}
 \put(110,39){$n-1$}
 \put(162,39){$n$}
 \put(50,-17){$2$}
 \end{picture}
 \end{equation*}
%
%
%
%
 \begin{equation*}
 \begin{picture}(150,75) 
 \put(-99,28){$F_4$}
 \put(-30,30){$\circ$}
 \put(-23,33){\line(1,0){32}}
 \put(10,30){$\circ$}
 \put(17,33){\line(1,0){32}}
 \put(50,30){$\circ$}
 \put(57,33){\line(1,0){32}}
 \put(90,30){$\circ$}
 \put(-30,20){$1$}
 \put(10,20){$2$}
 \put(50,20){$3$}
 \put(90,20){$4$}
 \put(30,35){$4$}
 \end{picture}
 \end{equation*}
%
%
 \begin{equation*}
 \begin{picture}(150,55) 
 \put(-99,28){$G_2$}
 \put(-30,30){$\circ$}
 \put(-23,33){\line(1,0){32}}
 \put(10,30){$\circ$}
 \put(-30,20){$1$}
 \put(10,20){$2$}
 \put(-10,35){$6$}
 \end{picture}
 \end{equation*}
\subsection{A distinguished double covering of Weyl groups}
\label{subsec:spinWeyl}

The Schur multipliers for finite Weyl groups $W$ (and actually for
all finite Coxeter groups) have been computed by Ihara and
Yokonuma \cite{IY} (also cf. \cite{Kar}). The explicit generators
and relations for the corresponding covering groups of $W$ can be
found in Karpilovsky \cite[Table 7.1]{Kar}.

We shall be concerned about a distinguished double covering
$\wtd{W}$ of $W$:
$$
1 \longrightarrow \Z_2 \longrightarrow \wtd{W} \longrightarrow W
\longrightarrow 1.
$$
We denote by $\Z_2 =\{1,z\},$ and by $\td{t}_i$ a fixed preimage
of the generators $s_i$ of $W$ for each $i$. The group $\wtd{W}$
is generated by $z, \td{t}_1,\ldots, \td{t}_n$ with relations

%
\begin{equation*}
z^2 =1, \qquad
 (\td{t}_{i}\td{t}_{j})^{m_{ij}} =
 \left\{
\begin{array}{rl}
1, & \text{if } m_{ij}=1,3  \\
z, & \text{if }  m_{ij}=2,4,6.
\end{array}
\right.
\end{equation*}


The quotient algebra $\C W^- :=\C \wtd{W} /\langle z+1\rangle$ of
$\C \wtd{W}$ by the ideal generated by $z+1$ will be called the
{\em spin Weyl group algebra} associated to $W$. Denote by $t_i
\in \C W^-$ the image of $\td{t}_i$. The spin Weyl group algebra
$\C W^-$ has the following uniform presentation: $\C W^-$ is the
algebra generated by $t_i, 1\le i\le n$, subject to the relations
\begin{equation}
(t_{i}t_{j})^{m_{ij}} = (-1)^{m_{ij}+1} \equiv \left\{
\begin{array}{rl}
1, & \text{if } m_{ij}=1,3  \\
-1, & \text{if }  m_{ij}=2,4,6.
\end{array}
\right.
\end{equation}
Note that $\dim \C W^- =|W|$. The algebra $\C W^-$ has a natural
superalgebra (i.e. $\Z_2$-graded) structure by letting each $t_i$
be odd.

By definition, the quotient by the ideal $\langle z-1\rangle$ of
the group algebra $\C \wtd{W}$ is isomorphic to $\C W$.

\begin{example}  \label{present}
Let $W$ be the Weyl group of type $A_{n},B_{n},$ or $D_{n}$, which
will be assumed in later sections. Then the spin Weyl group
algebra $\C W^-$ is generated by $t_1,\ldots, t_n$ with the
labeling as in the Coxeter-Dynkin diagrams and the explicit
relations summarized in the following table.
%

 \begin{center}
\begin{tabular}
[t]{|l|l|}\hline Type of $W$ & Defining Relations for $\C W^-$\\
\hline $A_{n}$  & $t_{i}^{2}=1$,
$t_{i}t_{i+1}t_{i}=t_{i+1}t_{i}t_{i+1}$,\\ &
$(t_{i}t_{j})^{2}=-1\text{ if } |i-j|\,>1$\\
\hline   & $t_{1},\ldots,t_{n-1}$ satisfy the relations for $\C W^-_{A_{n-1}}$, \\
$B_{n}$ &  $t_{n}^{2}=1,(t_{i}t_{n})^{2}=-1$ if $i\neq n-1,n$, \\ &
$(t_{n-1}t_{n})^{4}=-1$\\
\hline    & $t_{1},\ldots,t_{n-1}$ satisfy the relations for
$\C W^-_{A_{n-1}}$,\\
$D_{n}$&  $t_{n}^{2}=1,(t_{i}t_{n})^{2}=-1$ if $i\neq n-2, n$, \\&
$t_{n-2}t_{n}t_{n-2}=t_{n}t_{n-2}t_{n}$\\\hline
\end{tabular}
\end{center}
\end{example}

\bigskip%

\subsection{The Clifford algebra $\Cl_W$}
\label{subsec:cliff}




Denote by $\h$ the reflection representation of the Weyl group $W$
(i.e. a Cartan subalgebra of the corresponding complex Lie algebra
$\mathfrak g$). In the case of type $A_{n-1}$, we will always choose
to work with the Cartan subalgebra $\h$ of $gl_n$ instead of $sl_n$
in this paper.

Note that $\h$ carries a $W$-invariant nondegenerate bilinear form
$(-,-)$, which gives rise to an identification $\h^*\cong \h$ and
also a bilinear form on $\h^*$ which will be again denoted by
$(-,-)$. We identify $\h^*$ with a suitable subspace of $\C^N$ and
then describe the simple roots $\{\alpha_i\}$ for $\mathfrak g$
using a standard orthonormal basis $\{e_i\}$ of $\C^N$. It follows
that $(\alpha_i, \alpha_j) =-2\cos (\pi /m_{ij})$.
%
%
%
%

Denote by $\Cl_W$ the Clifford algebra associated to $(\h,
(-,-))$, which is regarded as a subalgebra of the Clifford algebra
$\Cl_N$ associated to $(\C^N,(-,-))$. We shall denote by $c_i$ the
generator in $\Cl_N$ corresponding to $\sqrt{2} e_i$ and denote by
$\be_i$ the generator of $\Cl_W$ corresponding to the simple root
$\al_i$ normalized with $\be_i^2=1$. In particular,
$\mathcal{C}_{N}$ is generated by $c_{1},\ldots,c_{N}$ subject to
the relations
\begin{align}  \label{clifford}
c_{i}^{2} =1,\quad c_{i}c_{j} =-c_{j}c_{i}\text{ if } i\neq j.
\end{align}

The explicit generators for $\Cl_W$ are listed in the following
table. Note that $\Cl_W$ is naturally a superalgebra with each
$\be_i$ being odd.

\bigskip%

 \begin{center}
\begin{tabular}
[t]{|l|l|p{3.5in}|}\hline Type of $W$&$N$ & Generators for
$\mathcal{C}_W$\\
\hline $A_{n-1}$ & $n$&
$\be_{i}=\frac{1}{\sqrt{2}}(c_{i}-c_{i+1}),1\leq i\leq n-1$\\
\hline $B_{n}$ &$n$&
$\be_{i}=\frac{1}{\sqrt{2}}(c_{i}-c_{i+1}),1\leq
i\leq n-1$, $\be_{n}=c_{n}$\\
\hline $D_{n}$ &$n$ &
$\be_{i}=\frac{1}{\sqrt{2}}(c_{i}-c_{i+1}),1\leq
i\leq n-1$, $\be_{n}=\frac{1}{\sqrt{2}}(c_{n-1}+c_{n})$\\
\hline
$E_{8}$ &8& $\be_{1}=\frac{1}{2\sqrt{2}}(c_{1}+c_{8}-c_{2}-c_{3}%
-c_{4}-c_{5}-c_{6}-c_{7})$\\ &&
$\be_{2}=\frac{1}{\sqrt{2}}(c_{1}+c_{2}),\be_{i}=\frac
{1}{\sqrt{2}}(c_{i-1}+c_{i-2}),3\leq i\leq8$\\
\hline $E_{7}$ &8& the subset of $\be_{i}$ in $E_{8}$, $1\le i\le
7$
\\
\hline $E_{6}$ &8& the subset of $\be_{i}$ in $E_{8}$, $1\le i\le
6$
\\
\hline $F_{4}$ &4& $\be_{1}=\frac{1}{\sqrt{2}}(c_{1}-c_{2}),\be%
_{2}=\frac{1}{\sqrt{2}}(c_{2}-c_{3})$\\
&& $\be_{3}=c_{3},\be_{4}=\frac{1}{2}(c_{4}-c_{1}-c_{2}%
-c_{3})$\\
\hline $G_{2}$ &3& $\be_{1}=\frac{1}{\sqrt{2}}(c_{1}-c_{2}),\be%
_{2}=\frac{1}{\sqrt{6}}(-2c_{1}+c_{2}+c_{3})$\\
\hline
\end{tabular}
\end{center}
\bigskip%

The action of $W$ on $\h$ and $\h^*$ preserves the bilinear form
$(-,-)$ and thus $W$ acts as automorphisms of the algebra $\Cl_W$.
This gives rise to a semi-direct product $\Cl_W \rtimes \C W$.
Moreover, the algebra $\Cl_W \rtimes \C W$ naturally inherits the
superalgebra structure by letting elements in $W$ be even and each
$\be_i$ be odd.

\subsection{The basic spin supermodule}

The following theorem is due to Morris \cite{Mo} in full
generality, and it goes back to I. Schur \cite{Sch} (cf.
\cite{Joz}) in the type $A$, namely the symmetric group case. It
can be checked  case by case using the explicit formulas of
$\be_i$ in the Table of Section~\ref{subsec:cliff}.

\begin{theorem} \label{th:morris}
Let $W$ be a finite Weyl group. Then, there exists a surjective
superalgebra homomorphism
$\C W^- \stackrel{\Omega}{\longrightarrow} \Cl_W$ which sends
$t_i$ to $\be_i$ for each $i$.
\end{theorem}

\begin{remark}
In \cite{Mo}, $W$ is viewed as a subgroup of the orthogonal Lie
group which preserves $(\h, (-,-))$. The preimage of $W$ in the
spin group which covers the orthogonal group provides the double
cover $\wtd{W}$ of $W$, where the Atiyah-Bott-Shapiro construction
of the spin group in terms of the Clifford algebra $\Cl_W$ was
used to describe this double cover of $W$.
\end{remark}

The superalgebra $\Cl_W$ has a unique (up to isomorphism) simple
supermodule (i.e. $\Z_2$-graded module). By pulling it back via the
homomorphism $\Omega:\C W^-\rightarrow \Cl_W$, we obtain a
distinguished $\C W^-$-supermodule, called the basic spin
supermodule. This is a natural generalization of the classical
construction for $\C S_n^-$ due to Schur \cite{Sch} (see
\cite{Joz}).
\subsection{A superalgebra isomorphism}

Given two superalgebras $\mathcal{A}$ and $\mathcal{B}$, we view the
tensor product of superalgebras $\mathcal{A}$ $\otimes$
$\mathcal{B}$ as a superalgebra with multiplication defined by
\begin{equation}
(a\otimes b)(a^{\prime}\otimes b^{\prime})
=(-1)^{|b||a^{\prime}|}(aa^{\prime }\otimes bb^{\prime})\text{ \ \
\ \ \ \ \ }(a,a^{\prime}\in\mathcal{A},\text{
}b,b^{\prime}\in\mathcal{B})
\end{equation}
where $|b|$ denotes the $\Z_2$-degree of $b$, etc. Also, we shall
use short-hand notation $ab$ for $(a\otimes b) \in \mathcal{A}$
$\otimes$ $\mathcal{B}$, $a = a\otimes1$, and $b=1\otimes b$.

We have the following Morita super-equivalence in the sense of
\cite{W2} between the superalgebras $\Cl_W \rtimes \C W$ and $\C
W^-$.

\begin{theorem} \label{th:isofinite}
We have an isomorphism of superalgebras:
$$\Phi: \Cl_W \rtimes \C W
\stackrel{\simeq}{\longrightarrow} \Cl_W \otimes \C W^-$$
which extends the identity map on $\Cl_W$ and sends $s_i \mapsto
-\sqrt{-1} \be_i t_i.$ The inverse map $\Psi$ is the extension of
the identity map on $\Cl_W$ which sends $ t_i \mapsto \sqrt{-1}
\be_i s_i.$
\end{theorem}

We first prepare some lemmas.

\begin{lemma}  \label{lem:braidmatch}
We have $(\Phi(s_i)\Phi(s_j))^{m_{ij}} = 1.$
\end{lemma}
\begin{proof}
Theorem~\ref{th:morris} says that $(t_it_j)^{m_{ij}} =
(\be_i\be_j)^{m_{ij}} = \pm 1$. Thanks to the identities $\be_j
t_i = -t_i\be_j$ and $\Phi(s_i) = -\sqrt{-1} \be_i t_i$, we have
\begin{align*}
(\Phi(s_i)&\Phi(s_j))^{m_{ij}} = (-\be_it_i \be_jt_j)^{m_{ij}}  \\
  &=(\be_i \be_j t_it_j)^{m_{ij}} =(\be_i \be_j)^{m_{ij}}
(t_it_j)^{m_{ij}} =1.
\end{align*}
\end{proof}

\begin{lemma} \label{lem:semidirect}
We have $\be_j \Phi(s_i) = \Phi(s_i) \, s_i(\be_j)$ for all $i,j$.
\end{lemma}
\begin{proof}
Note that $(\be_i,\be_i) =2\be_i^2=2$, and hence
$$\be_j \be_i =-\be_i\be_j +(\be_j,\be_i) =-\be_i\be_j
+\frac{2(\be_j,\be_i)}{(\be_i,\be_i)} \be_i^2 =-\be_i
s_i(\be_j).$$ Thus, we have
\begin{align*}
\be_j & \Phi(s_i) = -\sqrt{-1} \be_j \be_i t_i \\
 &= -\sqrt{-1} t_i \be_j \be_i
 = \sqrt{-1} t_i \be_i s_i(\be_j)
= \Phi(s_i) \, s_i(\be_j).
\end{align*}
\end{proof}

\begin{proof}[Proof of Theorem~\ref{th:isofinite}]
The algebra $\Cl_W \rtimes \C W$ is generated by $\be_i$ and $s_i$
for all $i$. Lemmas~\ref{lem:braidmatch} and \ref{lem:semidirect}
imply that $\Phi$ is a (super) algebra homomorphism. Clearly $\Phi$
is surjective, and thus an isomorphism by a dimension counting
argument.

Clearly, $\Psi$ and $\Phi$ are inverses of each other.
\end{proof}

\begin{remark}
The type $A$ case of Theorem~\ref{th:isofinite} was due to Sergeev
and Yamaguchi independently \cite{Ser, Yam}, and it played a
fundamental role in clarifying the earlier observation in the
literature (cf. \cite{Joz, St}) that the representation theories
of $\C S_n^-$ and $\Cl_n \rtimes \C S_n$ are essentially the same.
\end{remark}

In the remainder of the paper, $W$ is always assumed to be one of
the classical Weyl groups of type $A, B$, or $D$.

\section{Degenerate affine Hecke-Clifford algebras}
\label{sec:clifford}

In this section, we introduce the degenerate affine Hecke-Clifford
algebras of type $D$ and $B$, and establish some basic properties.
The degenerate affine Hecke-Clifford algebra associated to the
symmetric group $S_n$ was introduced earlier by Nazarov under the terminology of
the affine Sergeev algebra \cite{Naz}.


\subsection{The algebra $\aHC_W$ of type $A_{n-1}$}
\begin{definition} \cite{Naz}
Let $u\in \C$, and $W=W_{A_{n-1}} =S_n$ be the Weyl group of type
$A_{n-1}$. The degenerate affine Hecke-Clifford algebra of type
$A_{n-1}$, denoted by $\aHC_W$ or $\aHC_{A_{n-1}}$, is the
 algebra generated by $x_1,\ldots, x_n$,
$c_1,\ldots,c_n$, and $S_n$ subject to the relation (\ref{clifford})
and the following relations:
\begin{align}
x_i x_j & =x_j x_i \quad (\forall i,j) \label{polyn} \\
x_{i}c_{i} &  =-c_{i}x_{i},\text{ }x_{i}c_{j}=c_{j}x_{i}\quad (i\neq  j)  \label{xici}  \\
\sigma c_{i} &  =c_{\sigma  i}\sigma\quad (1\leq i\leq
n,\sigma \in S_{n}) \label{sigmac} \\
x_{i+1}s_{i}-s_{i}x_{i} &  =u(1-c_{i+1}c_{i})    \label{xisi} \\
x_{j}s_{i} &  =s_{i}x_{j}\quad (j\neq i,i+1)  \label{xjsi}
\end{align}
\end{definition}

\begin{remark}
Alternatively, we may view $u$ as a formal parameter and the
algebra $\aHC_W$ as a $\C(u)$-algebra. Similar remarks apply to
various algebras introduced in this paper. Our convention
$c_i^2=1$ differs from Nazarov's which sets $c_i^2=-1$.
\end{remark}

The symmetric group $S_n$ acts as the automorphisms on the
symmetric algebra $\C[\h^*] \cong \C [x_1,\ldots,x_n]$ by
permutation. We shall denote this action by $f \mapsto f^\sigma$
for $\sigma \in S_n, f \in \C [x_1,\ldots,x_n]$.

\begin{proposition} \label{mulA}
Let $W=W_{A_{n-1}}$. Given $f \in \C [x_1,\ldots,x_n]$ and $1 \leq i
\leq n-1$, the following identity holds in $\aHC_W$:
\[
s_i f = f^{s_i} s_i + u\displaystyle \frac{f - f^{s_i}}{x_{i+1} -
x_i} + u\displaystyle \frac{c_ic_{i + 1}f - f^{s_i} c_ic_{i +
1}}{x_{i+1} + x_i}.
\]
\end{proposition}
\noindent It is understood here and in similar expressions below
that $\frac{A}{g(x)} =\frac1{g(x)} \cdot A$. In this sense, both
numerators on the right-hand side of the above formula are
(left-)divisible by the corresponding denominators.

\begin{proof}
By the definition of $\aHC_W$, we have that $s_i x_j^k = x_j^k s_i$
for any $k$ if $j \neq i, i+1$. So it suffices to check the identity
for $f = x_i^k x_{i+1}^l$. We will proceed by induction.

First, consider $f = x_i^k$, i.e. $l=0$. For $k=1$, this follows
from (\ref{xisi}).
%
%
Now assume that the statement is true for $k$. Then
\begin{eqnarray*}
s_{i}x_{i}^{k+1}
&=& \left(  x_{i+1}^{k}s_{i}%
+u\frac{(x_{i}^{k}-x_{i+1}^{k})}{x_{i+1}-x_{i}}+u\frac{(c_ic_{i + 1}x_{i}%
^{k}-x_{i+1}^{k}c_ic_{i + 1})}{x_{i+1}+x_{i}}\right)  x_{i}\\
&=& x_{i+1}^{k}\left(  x_{i+1}s_{i}-u(1-c_{i+1}c_{i})\right)  \\
&& +u\frac{(x_{i}^{k}-x_{i+1}^{k})}{x_{i+1}-x_{i}}x_{i}
+u\frac{(c_ic_{i + 1}x_{i}^{k}-x_{i+1}^{k}c_ic_{i + 1})}{{x_{i+1}+x}_{i}}x_{i}{\ \ }\\
&=&
x_{i+1}^{k+1}s_{i}+u\frac{(x_{i}^{k+1}-x_{i+1}^{k+1})}{x_{i+1}-x_{i}
}+u\frac{(c_ic_{i + 1}x_{i}^{k+1}-x_{i+1}^{k+1}c_ic_{i +
1})}{x_{i+1}+x_{i}},
\end{eqnarray*}
where the last equality is obtained by using (\ref{xici}) and
(\ref{xisi}) repeatedly.

An induction on $l$ will complete the proof of the proposition for
the monomial $f=x_{i}^{k}x_{i+1}^{l}$. The case $l=0$ is
established above. Assume the formula is true for
$f=x_{i}^{k}x_{i+1}^{l}$. Then using $s_{i}x_{i+1}
=x_{i}s_{i}+u(1+c_{i+1}c_{i})$, we compute that 
\begin{eqnarray*}
s_{i}x_{i}^{k}x_{i+1}^{l+1} &=& \left(
x_{i}^{l}x_{i+1}^{k}s_{i}+u\frac{(x_{i}^{k}x_{i+1}^{l}
-x_{i}^{l}x_{i+1}^{k})}{x_{i+1}-x_{i}} \right. \\
&& \left.\qquad \qquad \quad +u\frac{(c_ic_{i +
1}x_{i}^{k}x_{i+1}^{l}
-x_{i}^{l}x_{i+1}^{k}c_ic_{i + 1})}{x_{i+1}+x_{i}}\right)\cdot  x_{i+1}\\
%
&=& x_{i}^{l}x_{i+1}^{k} (  x_{i}s_{i}+u(1+c_{i+1}c_{i}))  \\
&&
+u\frac{(x_{i}^{k}x_{i+1}^{l+1}-x_{i}^{l}x_{i+1}^{k+1})}{x_{i+1}-x_{i}}
+u\frac{(c_ic_{i + 1}
x_{i}^{k}x_{i+1}^{l+1}+x_{i}^{l}x_{i+1}^{k+1}
c_ic_{i + 1})}{{x_{i+1}+x}_{i}}{\ \ }\\
%
%
&=& x_{i}^{l+1}x_{i+1}^{k}s_{i}+u\frac{(x_{i}^{k}
x_{i+1}^{l+1}-x_{i}^{l+1}x_{i+1}^{k})}{x_{i+1}-x_{i}} \\
&& \qquad \qquad \quad +u\frac{(c_ic_{i +
1}x_{i}^{k}x_{i+1}^{l+1}-x_{i}^{l+1} x_{i+1}^{k}c_ic_{i +
1})}{x_{i+1}+x_{i}}.
\end{eqnarray*}
This completes the proof of the proposition.
\end{proof}

The algebra $\aHC_W$ contains $\C[\h^*], \Cl_n$, and $\C W$ as
subalgebras. We shall denote $x^\al =x_1^{a_1}\cdots x_n^{a_n}$ for
$\alpha =(a_1,\ldots, a_n) \in \Z_+^n$, $c^\epsilon
=c_1^{\epsilon_1}\cdots c_n^{\epsilon_n}$ for $\epsilon =
(\epsilon_1,\ldots,\epsilon_n) \in \Z_2^n.$

Below we give a new proof of the PBW basis theorem for $\aHC_W$
(which has been established by different methods in \cite{Naz,
Kle}), using in effect the induced $\aHC_W$-module
$\text{Ind}_W^{\aHC_W} {\bf 1}$ from the trivial $W$-module $\bf 1$.
This induced module is of independent interest. This approach will
then be used for type $D$ and $B$.

\begin{theorem} \label{PBW:A}
Let $W = W_{A_{n-1}}$. The multiplication of subalgebras $\C[\h^*],
\Cl_n$, and $\C W$ induces a vector space isomorphism
\[
\C[\h^*]\otimes \Cl_n \otimes \C
W\stackrel{\simeq}{\longrightarrow} \aHC_W.
\]
Equivalently, $\{x^\al c^\epsilon w| \al\in\Z_+^n, \epsilon
\in\Z_2^n,  w\in W\}$ forms a linear basis for $\aHC_W$ (called a
PBW basis).
\end{theorem}

\begin{proof}
Note that $\text{IND}:=\C [x_1,\ldots,x_n] \otimes \Cl_n$ admits an
algebra structure by (\ref{clifford}), (\ref{polyn}) and
(\ref{xici}). By the explicit defining relations of $\aHC_W$, we can
verify that the algebra $\aHC_W$ acts on $\text{IND}$ by letting
$x_i$ and $c_i$ act by left multiplication, and $s_i \in S_n$ act by
\begin{equation*}  \label{eq:induced}
s_i. (fc^\epsilon) = f^{s_i} c^{s_i\epsilon}+  \left (
u\displaystyle \frac{f - f^{s_i}}{x_{i+1} - x_i} + u\displaystyle
\frac{c_ic_{i + 1}f - f^{s_i} c_ic_{i + 1}}{x_{i+1} + x_i} \right
) c^\epsilon.
\end{equation*}

For $\alpha =(a_1,\ldots,a_n)$, we denote $|\alpha|
=a_1+\cdots+a_n$. Define a Lexicographic ordering $<$ on the
monomials $x^\alpha, \alpha \in \Z_+^n$, (or respectively on
$\Z_+^n$), by declaring $x^\al<x^{\al'}$, (or respectively
$\al<{\al'}$), if $|\al| <|\al'|$, or if $|\al| =|\al'|$ then
there exists an $1\le i \le n$ such that $a_i<a_i'$ and $a_j=a_j'$
for each $j<i$.

Note that the algebra $\aHC_W$ is spanned by the elements of the
form $x^\al c^\epsilon w$. It remains to show that these elements
are linearly independent.

Suppose that $S := \sum a_{_{\al \epsilon w}} x^\al c^\epsilon w
=0$ for a finite sum over $\al, \epsilon, w$ and that some
coefficient $a_{_{\al \epsilon w}} \neq 0$; we fix one such
$\epsilon$. Now consider the action $S$ on an element of the form
$x_1^{N_1} x_2^{N_2}\cdots x_n^{N_n}$ for $N_1 \gg N_2 \gg\cdots
\gg N_n \gg 0$. Let $\tilde{w}$ be such that $(x_1^{N_1}
x_2^{N_2}\cdots x_n^{N_n})^{\tilde{w}}$ is maximal among all
possible $w$ with $a_{\al \epsilon w} \neq 0$ for some $\al$. Let
$\tilde{\al}$ be the largest element among all $\alpha$ with
$a_{_{\al \epsilon \td{w}}} \neq 0$. Then among all monomials in
$S(x_1^{N_1} x_2^{N_2}\cdots x_n^{N_n})$, the monomial
$x^{\tilde{\al}} (x_1^{N_1} x_2^{N_2}\cdots x_n^{N_n})^{\tilde{w}}
c^\epsilon$ appears as a maximal term with coefficient $\pm
a_{{\tilde{\al} \epsilon \tilde{w}}}$. It follows from $S=0$ that
$a_{{\tilde{\al} \epsilon \tilde{w}}} =0$. This is
a contradiction, 
and hence the elements $x^\al c^\epsilon w$ are linearly
independent.
\end{proof}

\begin{remark}
By the PBW Theorem~\ref{PBW:A}, the $\aHC_W$-module $\text{IND}$
introduced in the above proof can be identified with the
$\aHC_W$-module induced from the trivial $\C W$-module. The same
remark applies below to type $D$ and $B$.
\end{remark}


\subsection{The algebra $\aHC_W$ of type $D_n$}
Let $W =W_{D_{n}}$ be the Weyl group of type $D_n$. It is
generated by $s_1,\ldots,s_n$, subject to the following relations:
\begin{align}
s_{i}^{2} &  =1\quad (i\leq n-1)   \label{eq:invol} \\
s_{i}s_{i+1}s_{i} &  =s_{i+1}s_{i}s_{i+1}\quad (i\leq n-2)  \\
s_{i}s_{j} &  =s_{j}s_{i}\quad (|i-j|>1,\; i,j\neq n)  \label{eq:comm} \\
s_{i}s_{n} &  =s_{n}s_{i}\quad (i\neq n-2) \\
s_{n-2}s_{n}s_{n-2} &  =s_{n}s_{n-2}s_{n}, \quad s_n^2=1.
\label{eq:braidD}
\end{align}
In particular, $S_n$ is generated by $s_1,\ldots,s_{n-1}$ subject
to the relations (\ref{eq:invol}--\ref{eq:comm}) above.

\begin{definition}
Let $u\in \C$, and let $W = W_{D_n}$. The degenerate affine
Hecke-Clifford algebra of type $D_{n}$, denoted by $\aHC_W$ or
$\aHC_{D_n}$, is the  algebra generated by $x_i, c_i, s_i$, $1\le
i\le n$, subject to the relations (\ref{polyn}--\ref{xjsi}),
(\ref{eq:invol}--\ref{eq:braidD}), and the following additional
relations:
\begin{align}
s_{n}c_{n} &  =-c_{n-1}s_{n} \nonumber \\
s_{n}c_{i} & =c_{i}s_{n}  \quad (i\neq n-1,n) \nonumber  \\
s_{n}x_{n} + x_{n-1}s_{n}&  = -u(1+c_{n-1}c_{n}) \label{Dsnxn} \\
s_n x_i & = x_i s_n   \quad (i\neq n-1,n).  \nonumber
\end{align}
\end{definition}


\begin{proposition}  \label{DinvolCliff}
The algebra $\aHC_{D_n}$ admits anti-involutions $\tau_1,\tau_2$
defined by
\begin{align*}
    \tau_1&: s_i\mapsto s_i, \quad\;\; c_j\mapsto c_j, \quad x_j\mapsto
    x_j,\quad (1\le i\le n);\\
    \tau_2&: s_i\mapsto s_i, \quad c_j\mapsto -c_j, \quad x_j\mapsto
    x_j,\quad (1\le i\le n).
\end{align*}
Also, the algebra $\aHC_{D_n}$ admits an involution $\sigma$ which
fixes all generators $s_i, x_i, c_i$ except the following 4
generators:
\[
\sigma: s_n\mapsto s_{n-1}, \quad s_{n-1}\mapsto s_n,
    \quad x_n\mapsto -x_{n}, \quad c_n\mapsto -c_n.
\]
\end{proposition}

\begin{proof}
We leave the easy verifications on $\tau_1, \tau_2$ to the reader.

It remains to check that $\sigma$ preserves the defining relations.
Almost all the relations are obvious except (\ref{xisi}) and
(\ref{Dsnxn}). We see that $\sigma$ preserves (\ref{xisi}) as
follows: for $i \leq n-2$,
\begin{eqnarray*}
    \sigma(x_{i+1}s_{i}-s_{i}x_{i}) &=&
    x_{i+1}s_{i}-s_{i}x_{i}\\
    &=& u(1-c_{i+1}c_{i})
   = \sigma(u(1-c_{i+1}c_{i})); \\
    \sigma(x_ns_{n-1}-s_{n-1}x_{n-1}) &=& -x_n s_n - s_n x_{n-1}\\
    &=& u(1+c_{n}c_{n-1})
    = \sigma(u(1-c_nc_{n-1})).
\end{eqnarray*}

Also, $\sigma$ preserves (\ref{Dsnxn}) since
\begin{eqnarray*}
    \sigma(s_n x_n + x_{n-1}s_n) &=& -s_{n-1}x_n + x_{n-1} s_{n-1}\\
    &=& -u(1-c_{n-1}c_{n})
    = \sigma(-u(1+c_{n-1}c_{n})).
\end{eqnarray*}
Hence, $\sigma$ is an automorphism of $\aHC_{D_n}$. Clearly
$\sigma^2 =1$.
\end{proof}

The natural action of $S_n$ on $\C[\h^*] =\C[x_1,\ldots,x_n]$ is
extended to an action of $W_{D_n}$ by letting
$$x_n^{s_n} = -x_{n-1}, \quad x_{n-1}^{s_n} = -x_{n}, \quad x_i^{s_n} =
x_{i} \quad (i \neq n-1,n).
$$

\begin{proposition}  \label{Didentity}
Let $W=W_{D_n}$,  $1\leq i\leq n-1$, and  $f \in \C
[x_1,\ldots,x_n]$. Then the following identities hold in $\aHC_W$:

\begin{enumerate}
\item $\displaystyle s_i f = f^{s_i} s_i + u\displaystyle \frac{f -
f^{s_i}}{x_{i+1} - x_i} + u\displaystyle \frac{c_ic_{i + 1}f -
f^{s_i} c_ic_{i + 1}}{x_{i+1} + x_i},$

\item $s_{n}f=f^{s_n}
s_{n} -u\displaystyle\frac{f-f^{s_n}}{x_{n}+x_{n-1}} +
u\displaystyle\frac{c_{n-1}c_{n}f-f^{s_n}c_{n-1}c_{n}}{x_{n}-x_{n-1}}$.
\end{enumerate}
\end{proposition}

\begin{proof}
Formula (1) has been established by induction as in type $A_{n-1}$.
Formula (2) can be verified by a similar induction.
\end{proof}


\subsection{The algebra $\aHC_W$ of type $B_n$}

Let $W =W_{B_n}$ be the Weyl group of type $B_n$, which is generated
by $s_1,\ldots,s_n$, subject to the defining relation for $S_n$ on
$s_1, \ldots, s_{n-1}$ and the following additional relations:
\begin{align}
s_i s_{n} &  =s_ns_i  \quad (1\le i \leq n-2) \label{eq:sisnB} \\
(s_{n-1}s_{n})^4 &  =1, \; s_{n}^{2} =1.  \label{braidB}
\end{align}

We note that the simple reflections $s_1,\ldots,s_n$ belongs to two
different conjugacy classes in $W_{B_n}$, with $s_1,\ldots,s_{n-1}$
in one and $s_n$ in the other.

\begin{definition}
Let $u,v\in \C$, and let $W=W_{B_n}$. The degenerate affine
Hecke-Clifford algebra of type $B_{n}$, denoted by $\aHC_W$ or
$\aHC_{B_n}$, is the  algebra generated by $x_i, c_i, s_i$, $1\le
i\le n$, subject to the relations (\ref{polyn}--\ref{xjsi}),
(\ref{eq:invol}--\ref{eq:comm}), (\ref{eq:sisnB}--\ref{braidB}), and
the following additional relations:
\begin{align*}
s_{n}c_{n} &  =-c_{n}s_{n} \\
s_{n}c_{i} &=c_{i}s_{n}\quad (i\neq n)   \\
s_{n}x_{n} +x_{n}s_{n}&  = -\sqrt{2} \,v\\
s_n x_i &=x_i s_n \quad  (i\neq n).
\end{align*}
\end{definition}
The factor $\sqrt{2}$ above is inserted for the convenience later in
relation to the spin affine Hecke algebras. When it is necessary to
indicate $u,v$, we will write $\aHC_W(u,v)$ for $\aHC_W$. For any $a
\in \C\backslash \{0\}$, we have an isomorphism of superalgebras
$\psi: \aHC_W(au,av) \rightarrow \aHC_W(u,v)$ given by dilations
$x_{i}\mapsto a x_{i}$ for $1\leq i\leq n$, while fixing each $s_i,
c_i$.

The action of $S_n$ on $\C[\h^*] =\C[x_1,\ldots,x_n]$ can be
extended to an action of $W_{B_n}$ by letting
$$
x_n^{s_n} = -x_{n}, \quad x_i^{s_n} = x_{i}, \quad (i \neq n).
$$

\begin{proposition}  \label{Bidentity}
Let $W=W_{B_n}$. Given $f\in \C [x_1,\ldots,x_n]$ and $1\le i\leq
n-1$, the following identities hold in $\aHC_W$:

\begin{enumerate}
\item $s_i f = f^{s_i} s_i + u\displaystyle \frac{f -
f^{s_i}}{x_{i+1} - x_i} + u\displaystyle \frac{c_ic_{i + 1}f -
f^{s_i} c_ic_{i + 1}}{x_{i+1} + x_i},$

\item $s_{n}f=f^{s_n}s_{n} -\sqrt{2} \displaystyle
v\frac{f-f^{s_n}}{2 x_{n}}$.

\end{enumerate}
\end{proposition}

\begin{proof}
The proof is similar to type $A$ and $D$, and will be omitted.
\end{proof}
\subsection{PBW basis for $\aHC_W$}

Note that $\aHC_W$ contains $\C[\h^*], \Cl_n, \C W$ as
subalgebras. We have the following PBW basis theorem for $\aHC_W$.

\begin{theorem}  \label{PBW:DB}
Let $W=W_{D_n}$ or $W=W_{B_n}$. The multiplication of subalgebras
$\C[\h^*], \Cl_n$, and $\C W$ induces a vector space isomorphism%
\[
\C[\h^*]\otimes \Cl_n \otimes \C W\longrightarrow\aHC_W.
\]
Equivalently, the elements $\{x^\al c^\epsilon w| \al\in\Z_+^n,
\epsilon \in\Z_2^n, w\in W\}$ form a linear basis for $\aHC_W$
(called a PBW basis).
\end{theorem}

\begin{proof}
For $W=W_{D_n}$, we can verify by a direct lengthy computation
that the $\aHC_{A_{n-1}}$-action on $\text{IND}=\C
[x_1,\ldots,x_n] \otimes \Cl_n$ (see the proof of
Theorem~\ref{PBW:A}) naturally extends to an action of
$\aHC_{D_n}$, where (compare Proposition~\ref{Didentity}) $s_n$
acts by
\[
s_{n} . (fc^\epsilon) =f^{s_n}c^{s_n\epsilon}
 - \left ( u\displaystyle\frac{f-f^{s_n}}{x_{n}+x_{n-1}}%
-u\displaystyle\frac{c_{n-1}c_{n}f -f^{s_n}
c_{n-1}c_{n}}{x_{n}-x_{n-1}} \right)
 c^\epsilon.
\]
Similarly, for $W=W_{B_n}$, the $\aHC_{A_{n-1}}$-action on
$\text{IND}$ extends to an action of $\aHC_{B_n}$, where (compare
Proposition~\ref{Bidentity}) $s_n$ acts by
\[
s_{n} . (fc^\epsilon) =f^{s_n}c^{s_n\epsilon}
 -\sqrt{2}  \displaystyle v \frac{f-f^{s_n}}{2
x_{n}} c^\epsilon.
\]

It is easy to show that, for either $W$, the elements $x^\al
c^\epsilon w$ 
span $\aHC_W$. It remains to show that they are linearly
independent. We shall treat the $W_{B_n}$ case in detail and skip
the analogous $W_{D_n}$ case.

To that end, we shall refer to the argument in the proof of
Theorem~\ref{PBW:A} with suitable modification as follows. The
$\tilde{w} =((\eta_1, \dots, \eta_n), \sigma) \in W_{B_n} =\{\pm
1\}^n \rtimes S_n$ may now not be unique, but the $\sigma$ and the
$\tilde{\al}$ are uniquely determined. Then, by the same argument
on the vanishing of a maximal term, we obtain that
$\sum_{\tilde{w}} a_{{\tilde{\al} \epsilon \tilde{w}}} x^\al
(x_1^{N_1} x_2^{N_2}\cdots x_n^{N_n})^{\tilde{w}} =0$, and hence,
$$
\sum_{(\eta_1, \dots, \eta_n)} a_{{\tilde{\al} \epsilon
\tilde{w}}} (-1)^{\sum_{i=1}^n \eta_i N_i} =0.
$$
By choosing $N_1,\ldots, N_n$ with different parities ($2^n$
choices) and solving the $2^n$ linear equations, we see that all
$a_{{\tilde{\al} \epsilon \tilde{w}}} =0$. This can also be seen
more explicitly by induction on $n$. By choosing $N_n$ to be even
and odd, we deduce that for a fixed $\eta_n$,
$\sum_{(\eta_1, \dots, \eta_{n-1}) \in \{\pm 1\}^{n-1}}
a_{{\tilde{\al} \epsilon \tilde{w}}} (-1)^{\sum_{i=1}^{n-1} \eta_i
N_i} =0$, which is the equation for $(n-1)$ $x_i$'s and the
induction applies.
%
\end{proof}
\subsection{The even center for $\aHC_W$}

The {\em even center} of a superalgebra $A$, denoted by $Z(A)$, is
the subalgebra of even central elements of $A$.

\begin{proposition}  \label{centerAff}
Let $W=W_{D_n}$ or $W=W_{B_n}$. The even center $Z(\aHC_W)$ of
$\aHC_W$ is isomorphic to $\C [x_1^2,\ldots,x_n^2]^{W}$.
\end{proposition}

\begin{proof}
We first show that every $W$-invariant polynomial $f$ in
$x_1^2,\ldots,x_n^2$ is central in $\aHC_W$. Indeed, $f$ commutes
with each $c_i$ by (\ref{xici}) and clearly $f$ commutes with each
$x_i$. By Proposition~\ref{Didentity} for type $D_n$ or
Proposition~\ref{Bidentity} for type $B_n$, $s_i f = f s_i$ for each
$i$. Since $\aHC_W$ is generated by $c_i, x_i$ and $s_i$ for all
$i$, $f$ is central in $\aHC_W$ and $\C [x_1^2,\ldots,x_n^2] ^W
\subseteq Z(\aHC_W)$.

On the other hand, take an even central element $C = \sum
a_{_{\al,\epsilon,w}} x^{\al} c^{\epsilon} w$ in $\aHC_W$. We
claim that $w=1$ whenever $a_{_{\al,\epsilon,w}}\neq 0$.
Otherwise, let $1 \neq w_0 \in W$ be maximal with respect to the
Bruhat ordering in $W$ such that $a_{_{\al,\epsilon,w_0}} \neq 0$.
Then $x_i^{w_0} \neq x_i$ for some $i$. By
Proposition~\ref{Didentity} for type $D_n$ or
Proposition~\ref{Bidentity} for type $B_n$, $x_i^2 C - C x_i^2$ is
equal to $\sum_{\alpha, \epsilon} a_{_{\al,\epsilon,w_0}} x^{\al}
(x_i^2 - (x_i^{w_0})^2) c^{\epsilon} w_0$ plus a linear
combination of monomials not involving $w_0$, hence nonzero. This
contradicts with the fact that $C$ is central. So we can write $C
= \sum a_{_{\al,\epsilon}} x^{\al} c^{\epsilon}$.

Since $x_i C = C x_i$ for each $i$, then (\ref{xici}) forces $C$ to
be in $\C [x_1,\ldots,x_n]$. Now by (\ref{xici}) and $c_i C = C c_i$
for each $i$ we have that $C\in \C [x_1^2,\ldots,x_n^2]$. Since $s_i
C = C s_i$ for each $i$, we then deduce from
Proposition~\ref{Didentity} for type $D_n$ or
Proposition~\ref{Bidentity} for type $B_n$ that $C \in \C
[x_1^2,\ldots,x_n^2]^W$.

This completes the proof of the proposition.
\end{proof}

\subsection{The intertwiners in $\aHC_W$}

In this subsection, we will define the intertwiners in the
degenerate affine Hecke-Clifford algebras $\aHC_W$.

The following intertwiners $\phi_i \in \aHC_W$ (with $u=1$) for
$W=W_{A_{n-1}}$ were introduced by Nazarov \cite{Naz} (also cf.
\cite{Kle}), where $1\le i \le n-1$:
\begin{eqnarray} \label{intertwinersA}
\phi_i = (x_{i+1}^2 -  x_i^2)s_i - u(x_{i+1}+x_i) - u
(x_{i+1}-x_i)c_ic_{i+1}.
\end{eqnarray}
A direct computation using (\ref{xisi}) provides another
equivalent formula for $\phi_i$:
\begin{eqnarray*}
\phi_i = s_i (x_i^2 -  x_{i+1}^2) + u(x_{i+1}+x_i) + u
(x_{i+1}-x_i)c_ic_{i+1}.
\end{eqnarray*}

We define the intertwiners $\phi_i \in \aHC_W$ for $W=W_{D_n}$
$(1\le i \le n)$ by the same formula (\ref{intertwinersA}) for
$1\le i \le n-1$ and in addition by letting
\begin{align} \label{intertwinersD}
\phi_n \equiv \phi_n^D = (x_n^2-x_{n-1}^2)s_n + u(x_n - x_{n-1}) -
u(x_n + x_{n-1})c_{n-1}c_n.
\end{align}

We define the intertwiners $\phi_i \in \aHC_W$ for $W=W_{B_n}$
$(1\le i \le n)$ by the same formula (\ref{intertwinersA}) for
$1\le i \le n-1$ and in addition by letting
\begin{align} \label{intertwinersB}
\phi_n \equiv \phi_n^B = 2x_n^2s_n + \sqrt{2}v x_n.
\end{align}

The following generalizes the type $A_{n-1}$ results of Nazarov
\cite{Naz}.
\begin{theorem} \label{intertwiner}
Let W be either $W_{A_{n-1}}$, $W_{D_n}$, or $W_{B_n}$. The
intertwiners $\phi_i$ (with $1\le i \le n-1$ for type $A_{n-1}$
and $1\le i \le n$ for the other two types) satisfy the following
properties:
\begin{enumerate}
\item $\phi_i^2 = 2u^2(x_{i+1}^2 + x_i^2) - (x_{i+1}^2 - x_i^2)^2
\quad (1 \leq i \leq n-1, \forall W);$ \label{intert1}

\item $\phi_n^2 = 2u^2(x_{n}^2 + x_{n-1}^2) - (x_{n}^2 -
x_{n-1}^2)^2$, for type $D_n;$ \label{intert2}

\item $\phi_n^2 = 4x_n^4 - 2v^2x_n^2$, for type $B_n$;
\label{intert3} \label{interbraided}

\item $\phi_i f = f^{s_i} \phi_i \quad (\forall f\in
\C[x_1,\ldots,x_n], \forall i, \forall W);$ \label{intert:poly}

\item $\phi_i c_j  = c_j^{s_i} \phi_i \quad (1\le j\le n, \forall
i, \forall W);$

\item $\underbrace{\phi_i \phi_j \phi_i \cdots}_{m_{ij}}=
\underbrace{\phi_j\phi_i \phi_j \cdots}_{m_{ij}}$.
\end{enumerate}
\end{theorem}

\begin{proof}
Part (1) follows by a straightforward computation and can also be
found in \cite{Naz} (with $u=1$). Part (2) follows from (1) by
applying the involution $\sigma$ defined in
Proposition~\ref{DinvolCliff}. Part (3) and (5) follow by a direct
verification.

Part (4) for $W_{A_{n-1}}$ follows from clearing the denominators
in the formula in Proposition~\ref{mulA} and then rewriting in
terms of $\phi_i$ as defined in (\ref{intertwinersA}). Similarly,
(4) for $W_{D_n}$ and $W_{B_n}$ follows by rewriting the formulas
given in Proposition~\ref{Didentity} in type $D$ and
Proposition~\ref{Bidentity} in type $B$, respectively.

It remains to prove (6) which is less trivial. Recall that
$$
\overbrace{s_i s_j s_i \cdots}^{m_{ij}} = \overbrace{s_j s_i s_j
\cdots}^{m_{ij}},
$$
(denoting this element by $w$). Let $\text{IND}$ be the subalgebra
of
    $\aHC_W$ generated by $\C[x_1,\ldots,x_n]$ and $\Cl_n$. Denote by $\leq$ the Bruhat
    ordering on $W$. Then we can write
    \[
        {\phi_i \phi_j \phi_i \cdots} = f w + \sum\limits_{u< w} p_{u,w} u
    \]
    for some $f \in \C[x_1,\ldots,x_n]$, and $p_{u,w} \in
    \text{IND}$. We may rewrite
    \[
        {\phi_i \phi_j \phi_i \cdots} = f w + \sum\limits_{u< w} r_{u,w}' \phi_{u}
    \]
where $\phi_u :=\phi_a\phi_b \cdots$ for any subword
$u=s_as_b\cdots$ of $w =s_i s_j s_i \cdots$, and $r_{u,w}'$ is in
some suitable localization of $\text{IND}$ with the central
element $\prod_{1\le k \le n} x_k^2 \prod_{1\le i< j\le n} (x_i^2
-x_j^2)\in \text{IND}$ being invertible. Note that such a
localization is a free module over the corresponding localized
ring of $\C[x_1, \ldots, x_n]$. We can then write
    \[
       {\phi_j\phi_i \phi_j \cdots} = f w + \sum\limits_{u< w} r_{u,w}'' \phi_{u}
    \]
with the same coefficient of $w$ as for ${\phi_i \phi_j \phi_i
\cdots}$, according to Lemma~\ref{lem:u=zero}. The difference
$\Delta := ({\phi_i \phi_j \phi_i \cdots} - {\phi_j\phi_i \phi_j
\cdots})$ is of the form
$$\Delta =\sum_{u< w} r_{u,w} \phi_{u}$$
for some $r_{u,w}$. Observe by (4) that $\Delta p = p^w \Delta$
for any $p\in \C[x_1,\ldots,x_n]$. Then we have
    \[
      \sum_{u< w}p^w r_{u,w} \phi_{u}  = p^w \Delta= \Delta p
      = \sum_{u< w} r_{u,w} \phi_{u} p
      = \sum_{u< w} r_{u,w} p^u \phi_{u}.
   \]
In other words, $(p^w -p^u) r_{u,w} =0$ for all $p\in
\C[x_1,\ldots,x_n]$ for each given $u<w$. This implies that
$r_{u,w}=0$ for each $u$, and $\Delta = 0$. This completes the
proof of (6) modulo Lemma~\ref{lem:u=zero} below.
\end{proof}

\begin{lemma} \label{lem:u=zero}
The following identity holds:
$$\underbrace{\phi^0_i \phi^0_j \phi^0_i \cdots}_{m_{ij}}=
\underbrace{\phi^0_j\phi^0_i \phi^0_j \cdots}_{m_{ij}}$$
where $\phi^0_i$ denotes the specialization $\phi_i|_{u=0}$ of
$\phi$ at $u=0$ (or rather $\phi_n^B|_{v=0}$ when $i=n$ in the
type $B_n$ case.)
\end{lemma}
\begin{proof}
Let $W=W_{B_n}$. For $1 \leq i \leq n-1$, $m_{i, i+1} =3$. So we
have
    \begin{eqnarray*}
        \phi^0_i \phi^0_{i+1} \phi^0_i
        &=& (x_{i+1}^2 -  x_i^2)s_i (x_{i+2}^2 - x_{i+1}^2)s_{i+1}(x_{i+1}^2 -
        x_i^2)s_i\\
        &=& (x_{i+1}^2 -  x_i^2)(x_{i+2}^2 -  x_{i}^2)(x_{i+2}^2 -
        x_{i+1}^2)s_i s_{i+1} s_i\\
        &=& (x_{i+2}^2 - x_{i+1}^2)(x_{i+2}^2 -  x_{i}^2)(x_{i+1}^2 -
        x_i^2)s_{i+1}s_i s_{i+1}\\
        &=& (x_{i+2}^2 -  x_{i+1}^2)s_{i+1}(x_{i+1}^2 - x_i^2)s_i(x_{i+2}^2 -
        x_{i+1}^2)s_{i+1}\\
        &=& \phi^0_{i+1} \phi^0_i\phi^0_{i+1}.
    \end{eqnarray*}

Note that $m_{ij} = 2$ for $j \neq i, i+1$; clearly, in this case,
$\phi^0_i \phi^0_j = \phi^0_j \phi^0_i.$

Noting that $m_{n-1, n}=4$, we have
    \begin{eqnarray*}
        \phi^0_{n-1} \phi^0_{n} \phi^0_{n-1} \phi^0_n
        &=& 4(x_{n}^2 -  x_{n-1}^2)s_{n-1} x_n^2s_n(x_{n}^2 - x_{n-1}^2)s_{n-1} x_n^2s_n
        \\
        &=& 4(x_{n}^2 -  x_{n-1}^2) x_{n-1}^2(x_{n-1}^2 -
        x_{n}^2)x_{n}^2 s_{n-1}s_ns_{n-1}s_n\\
        &=& 4x_{n}^2 (x_{n}^2 -  x_{n-1}^2) x_{n-1}^2(x_{n-1}^2 -
        x_{n}^2) s_ns_{n-1}s_ns_{n-1}\\
        &=& 4x_{n}^2 s_n (x_{n}^2 -  x_{n-1}^2)s_{n-1} x_{n}^2s_n(x_{n}^2 -
        x_{n-1}^2)s_{n-1}\\
        &=& \phi^0_{n} \phi^0_{n-1}\phi^0_{n} \phi^0_{n-1}.
    \end{eqnarray*}
    This completes the proof for type $B_n$.

    The similar proofs for types $A_{n-1}$ and $D_n$ are skipped.
\end{proof}

Theorem~\ref{intertwiner} implies that for every $w \in W$ we have
a well-defined element $\phi_w \in \aHC_{W}$ given by
$\phi_w = \phi_{i_1}\cdots\phi_{i_m}$ where $w = s_{i_1}\cdots
s_{i_m}$ is any reduced expression for $w$. These elements
$\phi_w$ should play an important role for the representation
theory of the algebras $\aHC_W$. It will be very interesting to
classify the simple modules of $\aHC_W$ and to find a possible
geometric realization. This was carried out by Lusztig \cite{Lu1,
Lu2, Lu3} for the usual degenerate affine Hecke algebra case.

\section{Degenerate spin affine Hecke algebras}
\label{sec:spin}

In this section we will introduce the degenerate spin affine Hecke
algebra $\saH_W$ when $W$ is the Weyl group of types $D_n$ or
$B_n$, and then establish the connections with the corresponding
degenerate affine Hecke-Clifford algebras $\aHC_W$. See \cite{W1}
for the type $A$ case.

\subsection{The skew-polynomial algebra}
We shall denote by $\Cl[b_1,\ldots,b_n]$ the $\C$-algebra generated
by $b_1,\ldots,b_n$ subject to the relations
$$
b_ib_j+b_jb_i =0\quad (i\neq j).
$$
This is naturally a superalgebra by letting each $b_i$ be odd. We
will refer to this as the {\em skew-polynomial algebra} in $n$
variables. This algebra has a linear basis given by $b^\alpha
:=b_1^{k_1}\cdots b_n^{k_n}$ for $\alpha =(k_1,\ldots,k_n) \in
\Z_+^n$, and it contains a polynomial subalgebra $\C[b_1^2,\ldots,
b_n^2]$.
\subsection{The algebra $\saH_W$ of type $D_n$}

Recall that the spin Weyl group $\C W^-$ associated to a Weyl
group $W$ is generated by $t_1,\ldots,t_n$ subject to the
relations as specified in Example~\ref{present}.

\begin{definition} \label{def:saHD}
Let $u \in \C$ and let $W=W_{D_n}$. The degenerate spin affine
Hecke algebra of type $D_n$, denoted by $\saH_W$ or $\saH_{D_n}$,
is the  algebra  generated by $\Cl[b_1,\ldots,b_n]$ and $\C W^-$
subject to the following relations:
\begin{align*}
t_{i}b_{i} + b_{i+1}t_i &  =u \quad (1\leq  i\leq n-1) \\
t_ib_j &=-b_{j}t_i \quad (j\neq i,i+1, \; 1\le i\leq n-1) \\
t_nb_n + b_{n-1}t_n & = u\\
t_nb_i &=-b_it_n \quad (i \neq n-1,n).
\end{align*}
\end{definition}

The algebra $\saH_W$ is naturally a superalgebra by letting each
$t_i$ and $b_i$ be odd generators. It contains the type $A_{n-1}$
degenerate spin affine Hecke algebra $\saH_{A_{n-1}}$ (generated
by $b_1,\ldots,b_n, t_1, \ldots, t_{n-1}$) as a subalgebra.

\begin{proposition}
The algebra $\saH_{D_n}$ admits anti-involutions $\tau_1, \tau_2$
defined by
\begin{align*}
\tau_1& : t_i\mapsto -t_i,  \quad b_i\mapsto -b_i \quad (1\le i
\le n); \\
\tau_2& : t_i\mapsto t_i,  \quad b_i\mapsto b_i \quad (1\le i \le
n).
\end{align*}
Also, the algebra $\saH_{D_n}$ admits an involution $\sigma$ which
swaps $t_{n-1}$ and $t_n$ while fixing all the remaining
generators $t_i, b_i$.
\end{proposition}
\begin{proof}
Note that we use the same symbols $\tau_1, \tau_2, \sigma$ to
denote the (anti-) involutions for $\saH_{D_n}$ and $\aHC_{D_n}$
in Proposition~\ref{DinvolCliff}, as those on $\saH_{D_n}$ are the
restrictions from those on $\aHC_{D_n}$ via the isomorphism in
Theorem~\ref{th:isomDB} below. The proposition is thus established
via the isomorphism in Theorem~\ref{th:isomDB}, or follows by a
direct computation as in the proof of
Proposition~\ref{DinvolCliff}.
\end{proof}

\subsection{The algebra $\saH_W$ of type $B_n$}

\begin{definition} Let $u,v\in\C$, and $W=W_{B_n}$.
The degenerate spin affine Hecke algebra of type $B_n$, denoted by
$\saH_W$ or $\saH_{B_n}$, is the  algebra  generated by
$\Cl[b_1,\ldots,b_n]$ and $\C W^-$ subject to the following
relations:
\begin{align*}
t_{i}b_{i} + b_{i+1}t_i &  =u \quad (1\leq  i\leq n-1) \\
t_ib_j &=-b_{j}t_i \quad (j\neq i,i+1, \; 1\le i\leq n-1) \\
t_nb_n + b_{n}t_n & =v\\
t_nb_i &=-b_{i}t_n \quad (i\neq n).
\end{align*}
\end{definition}

Sometimes, we will write $\saH_W(u,v)$ or $\saH_{B_n}(u,v)$ for
$\saH_W$ or  $\saH_{B_n}$ to indicate the dependence on the
parameters $u,v$.

\subsection{A superalgebra isomorphism}
\begin{theorem}  \label{th:isomDB}
Let $W=W_{D_n}$ or $W=W_{B_n}$. Then,
\begin{enumerate}
\item there exists an isomorphism of superalgebras
$$\Phi:\aHC_W {\longrightarrow }\Cl_n \otimes \saH_W$$
which extends the isomorphism $\Phi: \Cl_n \rtimes \C W
\longrightarrow \Cl_n \otimes \C W^-$ (in
Theorem~\ref{th:isofinite}) and sends
$x_{i}\longmapsto\displaystyle \sqrt{-2}c_{i}b_{i}$ for each $i;$
%

\item the inverse $\Psi:\mathcal{C}_{n}\otimes
\saH_W{\longrightarrow}\aHC_W$ extends $\Psi: \Cl_n \otimes \C W^-
\longrightarrow \Cl_n \rtimes \C W$ (in
Theorem~\ref{th:isofinite}) and sends $b_{i}\longmapsto
\displaystyle\frac{1}{\sqrt{-2}}c_{i}x_{i}$ for each $i$.
%
\end{enumerate}
\end{theorem}
Theorem~\ref{th:isomDB} also holds for $W_{A_{n-1}}$ (see
\cite{W1}).

\begin{proof}
We only need to show that $\Phi$ preserves the defining relations
in $\aHC_W$ which involve $x_i$'s.

Let $W=W_{D_n}$. Here, we will verify two such relations below.
The verification of the remaining relations is simpler and will be
skipped. For $1 \leq i \leq n-1$, we have
{\allowdisplaybreaks
\begin{align*}
\Phi(x_{i+1}s_i -s_ix_i) & = c_{i+1}b_{i+1}(c_{i}-c_{i+1})t_i -
(c_{i}-c_{i+1})t_i c_{i}b_{i}\\
& =(1-c_{i+1}c_i)b_{i+1} t_i + (1-c_{i+1}c_i)t_i b_i\\
& =u(1-c_{i+1}c_i), \\
\Phi(s_nx_n +x_{n-1}s_n) & = (c_{n-1}+c_{n})t_n c_{n}b_{n} +
c_{n-1}b_{n-1}(c_{n-1}+c_{n})t_n\\
& =-(1+c_{n-1}c_n) t_nb_{n} - (1+c_{n-1}c_n)b_{n-1}t_n\\
& =-u(1+c_{n-1}c_n).
\end{align*}
}

Now let $W=W_{B_n}$. For $1 \leq i \leq n-1$, as in the proof in
type $D_n$, we have $\Phi(x_{i+1}s_i -s_ix_i) =u(1-c_{i+1}c_i)$.
Moreover, we have
{\allowdisplaybreaks
\begin{align*}
\Phi(s_n x_n + x_n s_n) & =
\displaystyle\frac{\sqrt{-2}}{\sqrt{-1}}c_n t_n c_{n}b_{n} +
\displaystyle\frac{\sqrt{-2}}{\sqrt{-1}}c_{n}b_{n} c_{n}t_n\\
& =\sqrt{2}c_n t_n c_n b_n + \sqrt{2} c_n b_n c_n t_n\\
& =-\sqrt{2}(t_n b_n + b_n c_n) =-\sqrt{2}v, \\
\Phi(s_n x_j) & = \displaystyle\frac{\sqrt{-2}}{\sqrt{-1}}c_n t_n
c_{j}b_{j}
=\sqrt{2}c_n t_n c_j b_j \\
& =\sqrt{2}c_j c_n t_n b_j =\sqrt{2}c_j b_j c_n t_n =\Phi(x_j
s_n), \text{ for } j \neq n.
\end{align*}
}
Thus $\Phi$ is a homomorphism of (super)algebras. Similarly, we
check that $\Psi$ is a superalgebra homomorphism. Observe that
$\Phi$ and $\Psi$ are inverses on generators and hence they are
indeed (inverse) isomorphisms.
\end{proof}

\subsection{PBW basis for $\saH_W$}

Note that $\saH_W$ contains the skew-polynomial algebra $\Cl
[b_1,\ldots,b_n]$ and the spin Weyl group algebra $\C W^-$ as
subalgebras. We have the following PBW basis theorem for $\saH_W$.

\begin{theorem} \label{PBW:DBspin}
Let $W=W_{D_n}$ or $W=W_{B_n}$. The multiplication of the
subalgebras $\C W^-$ and $\Cl [b_1,\ldots,b_n]$ induces a vector
space isomorphism
\[
\Cl [b_1,\ldots,b_n] \otimes \C W^-
\stackrel{\simeq}{\longrightarrow} \saH_W.
\]
\end{theorem}
Theorem~\ref{PBW:DBspin} also holds for $W_{A_{n-1}}$ (see
\cite{W1}).

\begin{proof}
It follows from the definition that $\saH_W$ is spanned by the
elements of the form $b^{\al} \sigma$ where $\sigma$ runs over a
basis for $\C W^-$ and $\al \in \Z_+^n$. By
Theorem~\ref{th:isomDB}, we have an isomorphism
$\psi:\Cl_{n}\otimes \frakH^-_W{\longrightarrow}\aHC_W. $ Observe
that the image $\psi(b^{\al} \sigma)$ are linearly independent in
$\aHC_W$ by the PBW basis Theorem~\ref{PBW:DB} for $\aHC_W$. Hence
the elements $b^{\al} \sigma$ are linearly independent in
$\saH_W$.
\end{proof}

\subsection{The even center for $\saH_W$}

\begin{proposition}
Let $W=W_{D_n}$ or $W=W_{B_n}$. The even center of $\saH_W$ is
isomorphic to $\C [b_1^2,\ldots,b_n^2]^{W}$.
\end{proposition}
\begin{proof}

By the isomorphism $\Phi: \aHC_W \rightarrow \Cl_n \otimes \saH_W$
(see Theorems~\ref{th:isomDB}) and the description of the center
$Z(\aHC_W)$ (see Proposition~\ref{centerAff}), we have
$$Z(\Cl_n \otimes
\saH_W) =\Phi(Z(\aHC_W)) = \Phi(\C [x_1^2,\ldots,x_n^2]^W) =\C
[b_1^2,\ldots,b_n^2]^W.
$$
Thus, $\C [b_1^2,\ldots,b_n^2]^W \subseteq Z(\saH_W)$.

Now let $C \in Z(\saH_W)$. Since $C$ is even, $C$ commutes with
$\Cl_n$ and thus commutes with the algebra $\Cl_n \otimes \saH_W$.
Then $\Psi (C) \in Z(\aHC_W) =\C [x_1^2,\ldots,x_n^2]^W$, and thus,
$C =\Phi\Psi(C) \in \Phi(\C [x_1^2,\ldots,x_n^2]^W) =\C
[b_1^2,\ldots,b_n^2]^W$.
\end{proof}
In light of the isomorphism Theorem~\ref{th:isomDB}, the problem
of classifying the simple modules of the spin affine Hecke algebra
$\saH_W$ is equivalent to the classification problem for the
affine Hecke-Clifford algebra $\aHC_W$. It remains to be seen
whether it is more convenient to find the geometric realization of
$\saH_W$ instead of $\aHC_W$.

\subsection{The intertwiners in $\saH_W$}

The intertwiners $\I_i  \in \saH_W$ $(1\le i \le n-1)$ for
$W=W_{A_{n-1}}$ were introduced in \cite{W1} (with $u=1$):
\begin{eqnarray} \label{spintertwinersA}
\I_i  = (b_{i+1}^2 -  b_i^2)t_i - u(b_{i+1}-b_i).
\end{eqnarray}
The commutation relations in Definition~\ref{def:saHD} gives us
another equivalent expression for $\I_i$:
\begin{eqnarray*}
\I_i  = t_i (b_i^2 -  b_{i+1}^2) + u(b_{i+1}-b_i).
\end{eqnarray*}

We define the intertwiners $\I_i \in \saH_W$ for $W=W_{D_n}$
$(1\le i \le n)$ by the same formula (\ref{spintertwinersA}) for
$1\le i \le n-1$ and in addition by letting
\begin{align} \label{spintertwinersD}
\I_n \equiv \I_n^D = (b_{n}^2 - b_{n-1}^2)t_n -u(b_{n}-b_{n-1}).
\end{align}

Also, we define the intertwiners $\I_i \in \saH_W$ for $W=W_{B_n}$
$(1\le i \le n)$ by the same formula (\ref{spintertwinersA}) for
$1\le i \le n-1$ and in addition by letting
\begin{align} \label{spintertwinersB}
\I_n \equiv \I_n^B = 2b_n^2 t_n - v b_n.
\end{align}

\begin{proposition} \label{spincomm}
The following identities hold in $\saH_W$, for $W =W_{A_{n-1}}$,
$W_{B_n}$, or $W_{D_n}$:
\begin{enumerate}
\item $\I_i  b_i  = -b_{i+1} \I_i, \I_i  b_{i+1}  = -b_i \I_i,$
and $\I_i  b_j  = -b_j \I_i \; (j \neq i,i+1)$, for $1\le i \le
n-1, 1\le j \le n,$ and any $W$;

\noindent In addition,

\item $\I_n b_{n-1} = -b_{n}\I_n, \I_n b_n = -b_{n-1}\I_n,$ and
$\I_n b_i = -b_i\I_n \; (i \neq n-1,n)$, for type $D_n$;

\item   $\I_n b_{n} = -b_{n}\I_n$, and $\I_n b_i = -b_i\I_n \;
(i\neq n)$, for type $B_n$.
\end{enumerate}
\end{proposition} \label{prop:spintertA}
\begin{proof}
  (1) We first prove the case when $j =i$:
    \begin{eqnarray*}
        \I_i  b_i
        &=& (b_{i+1}^2 -  b_i^2)t_ib_i -
        u(b_{i+1}-b_i)b_i \\
        &=& (b_{i+1}^2 -  b_i^2)(-b_{i+1}t_i +u) -
        u(b_{i+1}b_i-b_i^2) \\
        &=& -b_{i+1}\left((b_{i+1}^2 -  b_i^2)t_i -
        u(b_{i+1}-b_i)\right)\\
        &=& -b_{i+1} \I_i.
    \end{eqnarray*}

The proof for $\I_i  b_{i+1}  = -b_i \I_i$ is similar and thus
skipped.
%

    For $j \neq i, i+1$, we have $t_i b_j = -b_j t_i$, and hence $\I_i b_j
    = -b_j \I_i$.

(2) We prove only the first identity. The proofs of the remaining
two identities are similar and will be skipped.
    \begin{eqnarray*}
    \I_n b_{n-1}
    &=& (b_{n}^2 - b_{n-1}^2)t_nb_{n-1} -u(b_{n}-b_{n-1})b_{n-1} \\
    &=& (b_{n}^2 - b_{n-1}^2)(-b_nt_n +u) -u(b_{n}b_{n-1}-b_{n-1}^2) \\
    &=& -b_{n}\left((b_{n}^2 - b_{n-1}^2)t_n -u(b_{n}-b_{n-1})\right)\\
    &=& -b_{n} \I_n.
    \end{eqnarray*}

The proof of (3) is analogous to (2), and is thus skipped.
\end{proof}

Recall the superalgebra isomorphism $\Phi:\aHC_W {\longrightarrow
}\Cl_n \otimes \saH_W$ defined in Section~\ref{sec:spin} and the
elements $\beta_i \in \Cl_n$ defined in Section~\ref{sec:finite}.

\begin{theorem} \label{isom:intertw}
Let W be either $W_{A_{n-1}}$, $W_{D_n}$, or $W_{B_n}$. The
isomorphism $\Phi: \aHC_W \longrightarrow \Cl_n \otimes \saH_W$
sends $\phi_i$ $\mapsto -2 \sqrt{-1}\be_i \I_i$ for each $i$. More
explicitly, $\Phi$ sends
\begin{align*}
 \phi_i &\longmapsto -\sqrt{-2}(c_i-c_{i+1}) \otimes \I_i \quad (1 \leq i \leq n-1); \\
\phi_n &\longmapsto -\sqrt{-2}(c_{n-1} +c_n) \otimes \I_n \quad \text{for type $D_n$}; \\
\phi_n &\longmapsto - 2\sqrt{-1} c_n \otimes \I_n \quad \text{for
type $B_n$}.
\end{align*}
\end{theorem}

\begin{proof}
   Recall that the isomorphism $\Phi$ sends $s_i \mapsto -\sqrt{-1} \be_i t_i,$
   $x_{i}\mapsto\displaystyle \sqrt{-2}c_{i}b_{i}$ for each $i$.
   So, for $1\leq i\leq n-1$, we have the following
   \begin{eqnarray*}
        \Phi(\phi_i)
        &=& \Phi \left ((x_{i+1}^2 -  x_i^2)s_i - u(x_{i+1}+x_i) - u
            (x_{i+1}-x_i)c_ic_{i+1}\right)\\
        &=& -\sqrt{-2}(c_i-c_{i+1})(b_{i+1}^2 - b_i^2)t_i -
        u\sqrt{-2}(c_{i+1}b_{i+1}-c_ib_i) \\
        && - u\sqrt{-2}(c_{i+1}b_{i+1}-c_ib_i)c_ic_{i+1}\\
        &=& -\sqrt{-2}(c_i - c_{i+1})\left((b_{i+1}^2 - b_i^2)t_i
        -u(b_{i+1}-b_i)\right)\\
        &=& -\sqrt{-2}(c_i-c_{i+1}) \otimes \I_i.
   \end{eqnarray*}
   Next for $\phi_n \in \aHC_{D_n}$, we have
   \begin{eqnarray*}
        \Phi(\phi_n)
        &=& \Phi \left ((x_n^2-x_{n-1}^2)s_n + u(x_n - x_{n-1}) - u(x_n +
            x_{n-1})c_{n-1}c_n\right)\\
        &=& -\sqrt{-2}(c_n+c_{n-1})(b_{n}^2 - b_{n-1}^2)t_n +
        u\sqrt{-2}(c_{n}b_{n}-c_{n-1}b_{n-1}) \\
        && - u\sqrt{-2}(c_{n}b_{n}-c_{n-1}b_{n-1})c_{n-1}c_{n}\\
        &=& -\sqrt{-2}(c_n + c_{n-1})\left((b_{n}^2 - b_{n-1}^2)t_n
        -u(b_{n}-b_{n-1})\right)\\
        &=& -\sqrt{-2}(c_{n-1} +c_n) \otimes \I_n.
   \end{eqnarray*}

We skip the computation for $\phi_n \in \aHC_{B_n}$ which is very
similar but less complicated.
\end{proof}

\begin{proposition}\label{prop:spintert}
The following identities hold in $\saH_W$, for $W=W_{A_{n-1}}$,
$W_{B_n}$, or $W_{D_n}$:
\begin{enumerate}
    \item $\I_i^2 = u^2(b_{i+1}^2+b_i^2)-(b_{i+1}^2 -
    b_i^2)^2,$ for $1 \leq i \leq n-1$ and every type of $W$. \label{spintert1}
    \item $\I_n^2 = u^2(b_{n}^2+b_{n-1}^2) -(b_{n}^2 -
    b_{n-1}^2)^2,
    \quad$ for type $D_n$. \label{spintert2}
    \item $\I_n^2 = 4b_n^4 - v^2 b_n^2, \quad$ for type $B_n$. \label{spintert3}
\end{enumerate}
\end{proposition}
\begin{proof}
It follows from the counterparts in Theorem~\ref{intertwiner} via
the explicit correspondences under the isomorphism $\Phi$ (see
Theorem~\ref{isom:intertw}). It can of course also be proved by a
direct computation.
%
\end{proof}

\begin{proposition}\label{spinterbraided}
For $W=W_{A_{n-1}}$, $W_{B_n}$, or $W_{D_n}$, we have
\begin{eqnarray*}
     \underbrace{\I_i \I_j \I_i \cdots}_{m_{ij}}
      & = (-1)^{m_{ij}+1} \underbrace{\I_j\I_i \I_j
      \cdots}_{m_{ij}}.
\end{eqnarray*}
\end{proposition}
\begin{proof}

By Theorem~\ref{th:morris}, we have
\begin{eqnarray*}
\underbrace{\beta_i \beta_j \beta_i \cdots}_{m_{ij}} & =
(-1)^{m_{ij}+1} \underbrace{\beta_j\beta_i \beta_j
\cdots}_{m_{ij}}.
\end{eqnarray*}

Now the statement follows from the above equation and
Theorem~\ref{intertwiner}~(6) via the correspondence of the
intertwiners under the isomorphism $\Phi$ (see
Theorem~\ref{isom:intertw}).
\end{proof}

\begin{remark}
Proposition~\ref{spincomm}, Theorem~\ref{isom:intertw}, and
Proposition~\ref{spinterbraided} for $\saH_{A_{n-1}}$ can be found
in \cite{W1}.
\end{remark}

\section{Degenerate covering affine Hecke algebras}
\label{sec:cover}

In this section, the degenerate covering affine Hecke algebras
associated to the  double covers $\wtd{W}$ of classical Weyl groups
$W$ are introduced. It has as its natural quotients the usual
degenerate affine Hecke algebras $\aH_W$ \cite{Dr, Lu1, Lu2} and the
spin degenerate affine Hecke algebras $\saH_W$ introduced by the
authors.

Recall the distinguished double cover $\wtd{W}$ of a Weyl group
$W$ from Section~\ref{subsec:spinWeyl}.
\subsection{The algebra $\caH_W$ of type $A_{n-1}$}

\begin{definition}
Let $W = W_{A_{n-1}}$, and let $u\in \C$. The degenerate covering
affine Hecke algebra of type $A_{n-1}$, denoted by $\caH_W$ or
$\caH_{A_{n-1}}$, is the algebra generated by
$\td{x}_1,\ldots,\td{x}_n$ and $z, \td{t}_1, \ldots,
\td{t}_{n-1}$, subject to the relations for $\wtd{W}$ and the
additional relations:
\begin{align}
    z\td{x}_i  & =\td{x}_i z,\quad z \text{ is central of order } 2   \label{central}\\
    \td{x}_i \td{x}_j & = z \td{x}_j \td{x}_i \quad (i\neq j) \\
    \td{t}_i \td{x}_j & = z \td{x}_j \td{t}_i \quad (j\neq i, i+1) \\
    \td{t}_i \td{x}_{i+1} & = z \td{x}_i \td{t}_i + u.
    \label{coverHecke}
\end{align}
\end{definition}
Clearly $\caH_W$ contains $\C \wtd{W}$ as a subalgebra.

\subsection{The algebra $\caH_W$ of type $D_n$}

\begin{definition}
Let $W = W_{D_n}$, and let $u\in \C$. The degenerate covering
affine Hecke algebra of type $D_n$, denoted by $\caH_W$ or
$\caH_{D_n}$, is the algebra generated by
$\td{x}_1,\ldots,\td{x}_n$ and $z, \td{t}_1, \ldots, \td{t}_n$,
subject to the relations (\ref{central}--\ref{coverHecke}) and the
following additional relations:
\begin{align*}
    \td{t}_n \td{x}_i & = z \td{x}_i \td{t}_n \quad (i\neq
    n-1,n)\\
    \td{t}_n \td{x}_n & =-\td{x}_{n-1} \td{t}_n + u.
\end{align*}
\end{definition}

\subsection{The algebra $\caH_W$ of type $B_n$}

\begin{definition}
Let $W = W_{B_n}$, and let $u,v \in \C$. The degenerate covering
affine Hecke algebra of type $B_n$, denoted by $\caH_W$ or
$\caH_{B_n}$, is the algebra generated by
$\td{x}_1,\ldots,\td{x}_n$ and $z, \td{t}_1, \ldots, \td{t}_n$,
subject to the relations (\ref{central}--\ref{coverHecke}) and the
following additional relations:
\begin{align*}
    \td{t}_n \td{x}_i & = z \td{x}_i \td{t}_n \quad (i\neq n) \\
    \td{t}_n \td{x}_n & =-\td{x}_n \td{t}_n + v.
\end{align*}
\end{definition}
\subsection{PBW basis for $\caH_W$}

\begin{proposition} \label{quotient}
Let $W = W_{A_{n-1}}, W_{D_n}$, or $W_{B_n}$. Then the quotient of
the covering affine Hecke algebra $\caH_W$ by the ideal $\langle
z-1 \rangle$ (respectively, by the ideal  $\langle z+1 \rangle$)
is isomorphic to the usual degenerate affine Hecke algebras
$\aH_W$ (respectively, the spin degenerate affine Hecke algebras
$\saH_W$).
\end{proposition}

\begin{proof}
Follows by the definitions in terms of generators and relations of
all the algebras involved.
\end{proof}

\begin{theorem} \label{PBW:coveringAff}
Let $W = W_{A_{n-1}}, W_{D_n}$, or $W_{B_n}$. Then the elements
$\td{x}^{\al} \td{w}$, where $\al \in \Z_{+}^n \text{ \ and
}\td{w} \in \wtd{W}$, form a basis for $\caH_W$ (called a PBW
basis).
\end{theorem}

\begin{proof}
By the defining relations, it is easy to see that the elements
$\td{x}^{\al} \td{w}$ form a spanning set for $\caH_W$. So it
remains to show that they are linearly independent.

For each element $ t \in W$, denote the two preimages in $\wtd{W}$
of $t $ by $\{\td{t}, z\td{t}\}$. Now suppose that
\[
 0 = \sum a_{\al,\td{t}}\td{x}^{\al}\td{t} + b_{\al,\td{t}}
 z\td{x}^{\al}\td{t}.
 \]
Let $I^+$ and $I^-$ be the ideals of $\caH_W$ generated by $z-1$
and $z+1$ respectively. Then by Proposition~\ref{quotient},
$\caH_W/I^+ \cong \aH_W$ and $\caH_W/I^- \cong \saH_W$. Consider
the projections:
\[
\Upsilon_+: \caH_W \longrightarrow \caH_W/I^+, \qquad
\Upsilon_-: \caH_W \longrightarrow \caH_W/I^-.
\]

By abuse of notation, denote the image of $\td{x}^\al$ in $\aH_W$
by $x^\al$. Observe that
\begin{align*}
0 = \Upsilon_+ \left (\sum (a_{\al,\td{t}}\td{x}^{\al}\td{t} +
b_{\al,\td{t}}\td{x}^{\al}z\td{t}) \right)
 = \sum  (a_{\al,\td{t}} + b_{\al,\td{t}}) x^{\al}t\in \aH_W.
\end{align*}
Since it is known \cite{Lu1} that $\{ x^{\al} t| \al \in \Z_{+}^n
\text{ \ and } t\in W\}$ form a basis for the usual degenerate
affine Hecke algebra $\aH_W$, $a_{\al,\td{t}} =-b_{\al,\td{t}}$ for
all $\al$ and $t$. Similarly, denoting the image in $\C W^-$ of
$\td{t}$ by $\bar{t}$, we have
\begin{align*}
0 = \Upsilon_- \left (\sum (a_{\al,\td{t}}\td{x}^{\al}\td{t} +
b_{\al,\td{t}}\td{x}^{\al}z\td{t}) \right)
= \sum ( a_{\al,\td{t}} - b_{\al,\td{t}}) x^{\al}\bar{t}\in
\saH_W.
\end{align*}
Since $\{ x^{\al} \bar{t}\}$ is a basis for the spin degenerate
affine Hecke algebra $\saH_W$, we have $a_{\al,\td{t}}
=b_{\al,\td{t}}$ for all $\al$ and $t$. Hence, $a_{\al,\td{t}} =
b_{\al,\td{t}} =0$, and the linear independence is proved.
\end{proof}

\end{document}